\documentclass{amsart}
\usepackage{amssymb,latexsym,enumerate,verbatim,url}
\usepackage[l2tabu,orthodox]{nag}
\usepackage{tikz}
\newtheorem{Thm}[equation]{Theorem}
\newtheorem{Prop}[equation]{Proposition}
\newtheorem{Cor}[equation]{Corollary}
\newtheorem{Lem}[equation]{Lemma}
\newtheorem{Rem}[equation]{Remark}
\newtheorem{Conj}[equation]{Conjecture}
\numberwithin{equation}{section}
\DeclareMathOperator{\Ad}{Ad}
\DeclareMathOperator{\Cr}{C_r^*}
\begin{document}

\title[Linear dependency and square integrable representations]{Linear Dependency of translations and square integrable representations}

\author[P.A. Linnell]{Peter A. Linnell}
\address{Department of Mathematics \\
Virginia Tech \\
Blacksburg, VA. 24061-0123 \\
USA}
\email{plinnell@math.vt.edu}

\author[M. J. Puls]{Michael J. Puls}
\address{Department of Mathematics \\
John Jay College-CUNY \\
524 West 59th Street \\
New York, NY 10019 \\
USA}
\email{mpuls@jjay.cuny.edu}

\author[A. Roman]{Ahmed Roman}
\address{Department of Physics\\
Emory University\\
Atlanta, GA 30322\\
USA}
\email{ahmed.hemdan.roman@emory.edu}


\begin{abstract}
Let $G$ be a locally compact group. We examine the problem of determining when nonzero functions in $L^2(G)$ have linearly independent left translations. In particular, we establish some results for the case when $G$ has an irreducible, square integrable, unitary representation. We apply these results to the special cases of the affine group, the shearlet group and the Weyl-Heisenberg group. We also investigate the case when $G$ has an abelian, closed subgroup of finite index.
\end{abstract}


\date{Fri Jan 27 15:57:46 EST 2017}
\maketitle

\section{Introduction}\label{Introduction}
Let $G$ be a locally compact Hausdorff group with left invariant Haar
measure $\mu$. Denote by $L^p(G)$ the set of complex-valued functions
on $G$ that are $p$-integrable with respect to $\mu$, where $1 < p
\in \mathbb{R}$. As usual, identify functions in $L^p(G)$ that differ
only on a set of $\mu$-measure zero. We shall write $\Vert \cdot
\Vert_p$ to indicate the usual $L^p$-norm on $L^p(G)$. The {\em
regular representation} of $G$ on $L^p(G)$ is given by $L(g) f(x) =
f(g^{-1}x)$, where $g, x \in G$ and $f \in L^p(G)$. The function
$L(g) f$ is known as the left translation of $f$ by $g$ (many papers
use the word ``translate" instead of ``translation").  In
\cite{Rosenblatt08} Rosenblatt investigated the problem of
determining when the left translations of a nonzero function $f$ in
$L^2(G)$ are linearly independent. In other words, when can there be
a nonzero function $f \in L^2(G)$, some nonzero complex constants
$c_k$, and distinct elements $g_k \in G$, where $1 \leq k \leq n, k
\in \mathbb{N}$ (the positive integers)
such that
\begin{equation} \label{eq:leftdepequ} 
\sum_{k=1}^n c_k L(g_k) f = 0?
\end{equation}
It was shown in the introduction of \cite{Rosenblatt08} that if $G$
has a nontrivial element of finite order, then there is a nonzero
element in $L^2(G)$ that has a linear dependency among its left
translations. Thus, when trying to find nontrivial functions that
satisfy \eqref{eq:leftdepequ} it is more interesting to consider
groups for which all nonidentity elements have infinite order. In the
case of $G= \mathbb{R}^n$ it is known that every nonzero function in
$L^2(\mathbb{R}^n)$ has no linear dependency among its left
translations. Rosenblatt attacked \eqref{eq:leftdepequ} by trying  to determine if there is a relationship between the linear independence of the translations of functions in $L^2(G)$ and the linear independence of an element and its images under the action of $G$ in an irreducible representation of $G$. In order to gain insights into possible connections between these concepts, he computed examples for specific groups. The particular groups that he studied in \cite{Rosenblatt08} were the Heisenberg group and the affine group. What made these groups appealing is that they have irreducible representations that are intimately related to a time-frequency equation. Recall that an equation of the form
\begin{equation} \label{eq:timefreqequ}
\sum_{k=1}^n c_k \text{exp}(ib_k h(t)) f(a_k +t) =0,
\end{equation}
is a time-frequency equation, where $a_k, b_k \in \mathbb{R}, f \in L^2(\mathbb{R})$ and $h \colon \mathbb{R} \rightarrow \mathbb{R}$ is a nontrivial function. The case $h(t) = t$ corresponds to the Heisenberg group and $h(t) = e^t$ corresponds to the affine group. 

Now suppose that $G$ is a group that has an irreducible representation related to \eqref{eq:timefreqequ}. Rosenblatt wondered if there existed a nontrivial $f \in
L^2(\mathbb{R})$ that satisfied equation \eqref{eq:timefreqequ}, then this $f$ could be used to produce a nonzero $F \in L^2(G)$ with a linear dependency among its left translations. He then showed  \cite[Proposition 3.1]{Rosenblatt08} that there exists a nonzero $f \in L^2(\mathbb{R})$ that satisfies the following time-frequency equation
\begin{equation} \label{eq:timefreqaffine}
Cf(t) = f(t - \log 2) + \text{exp}(-\frac{i}{2} e^t)f(t - \log 2),
\end{equation}
where $C$ is a constant. This time-frequency equation corresponds to
the affine group $A$ case since $h(t) = e^t$. This offers some hope
that there might be a nonzero function in $L^2(A)$ that has a linear
dependency among its left translations. However, there is no clear
principle that can be used to show the existence of such a function
given a nontrivial $f$ that satisfies \eqref{eq:timefreqaffine}.
Using the proof of the existence of $f$ that satisfies \eqref{eq:timefreqaffine} as a guide, a nonzero $F$ in $L^2(A)$ with linearly dependent left translations was shown to exist \cite[Proposition 3.2]{Rosenblatt08}.

Even less is known about the Heisenberg group $H_n, n \in \mathbb{N}$. The relevant time-frequency equation, which has been intensely studied in the context of Gabor analysis,  is
\begin{equation} \label{eq:timefreqheisen}
\sum_{k=1}^m c_k e^{2\pi i b_k \cdot t} f( t + a_k) = 0,
\end{equation}
where $c_k$ are nonzero constants, $a_k, b_k \in \mathbb{R}^n$, and  $f \in L^2(\mathbb{R}^n)$. Linnell showed that $f=0$ is the only solution to \eqref{eq:timefreqheisen} when the subgroup generated by $(a_k, b_k)$, where $k =1, \dots, m$, is discrete. This gave a partial answer to a conjecture posed by Heil, Ramanathan and Topiwala on page 2790 of
\cite{HeilRamaTop96} that $f=0$ is the only solution to
\eqref{eq:timefreqheisen} when $n =1$. As far as we know the conjecture is still open. 

The motivation for this paper is to give a clearer picture of the link between the linear independence of an element and its images under the action of $G$ in an irreducible representation of $G$ and the linear independence of the left translations of a function in $L^2(G)$. In Section \ref{Proofofconnectingprop} we will prove the following:
\begin{Prop} \label{connectingprop}
Let $G$ be a locally compact group and suppose $\pi$ is an irreducible, unitary, square integrable representation of $G$ on a Hilbert space $\mathcal{H}_{\pi}$. If there exists a nonzero $v$ in $\mathcal{H}_{\pi}$ such that
\[ \sum_{k=1}^n c_k \pi(g_k) v = 0 \]
for some nonzero constants $c_k \in \mathbb{C}$ and $g_k \in G$, then there exists a nonzero $F \in L^2(G)$ that satisfies
\[ \sum_{k=1}^n c_k L(g_k) F = 0. \]
In particular if there exists a nonzero $v$ in $\mathcal{H}_{\pi}$ with linearly dependent translations, then there exists a nonzero $F$ in $L^2(G)$ with linearly dependent translations.
\end{Prop}

In Section \ref{affinegroup} we will use Proposition \ref{connectingprop} to construct explicit examples of nontrivial functions in $L^2(A)$, where $A$ is the affine group, that have a linear dependency among their left translations. 

In Section \ref{discretegroupzerodiv} we investigate the case where $G$ is a discrete group. We will see that the problem of determining if the left translations of a nonzero function in $\ell^2(G)$ forms a linearly independent set is related to the strong Atiyah conjecture. 
We shall briefly review the strong Atiyah conjecture in Section
\ref{discretegroupzerodiv}.

After considering the discrete group case in Section \ref{discretegroupzerodiv}, we shall return to studying the linear dependency problem for groups that satisfy our original hypotheses. Let $K$ be a subgroup of a group $G$. If $k \in K, x \in G$, and $f \in L^2(G)$, then we shall say that 
\[ L(k)f(x) = f(k^{-1}x) \]
is a {\em left $K$-translation} of $f$. In Section \ref{proofofdiscretesubthm} we shall prove

\begin{Thm} \label{discretesubgroup}
Let $G$ be a locally compact, $\sigma$-compact group and let $K$ be a
torsion-free discrete subgroup of $G$. If $K$ satisfies the strong Atiyah conjecture,
then each nonzero function in $L^2(G)$ has linearly independent $K$-translations.
\end{Thm}

In Section \ref{WeylHeisenberggroup} we shall study the Weyl-Heisenberg group $\tilde{H}_n$, a variant of the Heisenberg group, $H_n$. The group $\tilde{H}_n$ is of interest to us because it has an irreducible unitary representation on $L^2(\mathbb{R}^n)$, the Schr\"{o}dinger representation, which is square integrable. Furthermore, the time-frequency equation
\eqref{eq:timefreqheisen} is related to the Schr\"{o}dinger
representation. Now if $K$ is a torsion-free discrete subgroup of
$\tilde{H}_n$, then by Theorem \ref{discretesubgroup} every nonzero
element in $L^2(\tilde{H}_n)$ has linearly independent left
$K$-translations. It will then follow from Proposition
\ref{connectingprop} that if the subgroup of $\mathbb{R}^{2n}$
generated by $(a_k, b_k), 1 \leq k \leq m$, is discrete and  
the product $ a_h\cdot b_k \in \mathbb{Q}$ for all $h,k$, then $f = 0$
is the only solution to \eqref{eq:timefreqheisen}, see Proposition \ref{timefreq}. 
This gives a new proof of a special case of \cite[Proposition 1.3]{Linnell98} and 
sheds new insights on the problem.

In Section \ref{shearletgroups} we consider the problem of determining the linear independence of the left translations of a function in $L^2(S)$, where $S$ is the shearlet group. By using Proposition \ref{connectingprop} we will see that this problem is related to the question of determining the linear independence of a shearlet system of a function in $L^2(\mathbb{R}^2)$, which was recently studied in \cite{MaPetersen}.

In the last section of the paper we investigate the linear
independence of left translations of functions in $L^p(G)$ for
virtually abelian groups $G$ with no nontrivial compact subgroups. In particular, we generalize \cite[Theorem 1.2]{EdgarRosenblatt79}.

\section{Proof of Proposition \ref{connectingprop}} \label{Proofofconnectingprop}
In this section we will prove Proposition \ref{connectingprop}. Before we give our proof we will give some necessary definitions. A {\em unitary representation} of $G$ is a homomorphism $\pi$ from $G$ into the group $U(\mathcal{H}_{\pi})$ of unitary operators on a nonzero Hilbert space $\mathcal{H}_{\pi}$ that is continuous with respect to the strong operator topology. This means that $\pi \colon G \rightarrow U(\mathcal{H}_{\pi})$ satisfies $\pi(xy) = \pi(x) \pi(y), \pi(x^{-1}) = \pi(x)^{-1} = \pi(x)^{\ast}$, and $x \rightarrow \pi(x)u$ is continuous from $G$ to $\mathcal{H}_{\pi}$ for each $u \in \mathcal{H}_{\pi}$. A closed subspace $W$ of $\mathcal{H}_{\pi}$ is said to be {\em invariant} if $\pi(x) W \subseteq W$ for all $x \in G$. If the only invariant subspaces of $\mathcal{H}_{\pi}$ are $\mathcal{H}_{\pi}$ and $0$, then $\pi$ is said to be an {\em irreducible representation} of $G$. A representation that is not irreducible is defined to be a {\em reducible representation}. If $\pi_1$ and $\pi_2$ are unitary representations of $G$, an {\em intertwining operator} for $\pi_1$ and $\pi_2$ is a bounded liner map $T\colon \mathcal{H}_{\pi_1} \rightarrow \mathcal{H}_{\pi_2}$ that satisfies $T\pi_1(g) = \pi_2(g)T$ for all $g \in G$. We will assume throughout this paper that the inner product $\langle \cdot, \cdot \rangle$ on $\mathcal{H}_{\pi}$ is conjugate linear in the second component. If $u,v \in \mathcal{H}_{\pi}$, a {\em matrix coefficient} of $\pi$ is the function $F_{v, u} \colon G \rightarrow \mathbb{C}$ defined by 
\[ F_{v, u}(x) = \langle v, \pi(x)u \rangle. \]
We will indicate the $F_{u, u}$ case by $F_u$. An irreducible representation $\pi$ is said to be {\em square integrable} if there exists a nonzero $u \in \mathcal{H}_{\pi}$ such that $F_u \in L^2(G)$. We shall say that $u \in \mathcal{H}_{\pi}$ is {\em admissible} if $F_u \in L^2(G)$. The set of admissible elements in $\mathcal{H}_{\pi}$ will be denoted by $\Ad(\mathcal{H}_{\pi})$. A consequence of $\pi$ being irreducible is that if there is a nonzero admissible element in $\mathcal{H}_{\pi}$, then $\Ad(\mathcal{H}_{\pi})$ is dense in $\mathcal{H}_{\pi}$. In fact, $\Ad(\mathcal{H}_{\pi}) = \mathcal{H}_{\pi}$ if $G$ is unimodular, in addition to $\Ad(\mathcal{H}_{\pi})$ containing a nonzero element. By \cite[Theorem 3.1]{GrossmanMorletPaul85} there exists a self adjoint positive operator $C \colon \Ad(\mathcal{H}_{\pi}) \rightarrow \mathcal{H}_{\pi}$ such that if $u \in \Ad(\mathcal{H}_{\pi})$ and $v \in \mathcal{H}_{\pi}$, then
\begin{align*}
 \int_G \vert \langle v, \pi(x)u \rangle \vert^2\, d\mu & = \int_G \langle v, \pi(x) u \rangle \langle \overline{ v, \pi(x)u } \rangle \, d\mu \\
                                                                                     & = \Vert Cu \Vert^2 \Vert v \Vert^2,
\end{align*}
where $\Vert \cdot \Vert$ denotes the $\mathcal{H}_{\pi}$-norm. Thus if $u \in \Ad(\mathcal{H}_{\pi})$, then $F_{v, u} \in L^2(G)$ for all $v \in \mathcal{H}_{\pi}$.

We now prove Proposition \ref{connectingprop}. Suppose there exists a nonzero $v \in \mathcal{H}_{\pi}$ for which there exists a linear dependency among some of the elements $\pi(g)v$, where $g \in G$. So there exists nonzero constants $c_1, c_2, \dots, c_n$ and elements $g_1, g_2, \dots, g_n$ from $G$ with $\pi(g_j) \neq \pi(g_k)$ if $j \neq k$ such that 
\begin{equation}\label{eq:hypprop}
\sum_{k=1}^n c_k \pi(g_k) v = 0. 
\end{equation}
Let $u \in \Ad(\mathcal{H}_{\pi})$. Then $0 \neq F_{v,u} \in L^2(G)$.
Since $\pi$ is unitary,
$\langle \pi (g)v, \pi (x) u\rangle = \langle v, \pi (g^{-1}x)u\rangle$ for all $x$ and $g$ in $G$; in other words the continuous linear map
$v \mapsto F_{v, u} \colon \mathcal{H} \to L^2(G)$
intertwines $\pi$ with the regular representation $L$. Combining this observation with our hypothesis (\ref{eq:hypprop}) yields for all $x \in G$,
\[ \sum_{k=1}^n c_k L(g_k) F_{v,u}(x) =0, \] 
that is $F_{v, u}$ has linearly dependent left translations. The proof of Proposition \ref{connectingprop} is now complete.

\section{The Affine Group}\label{affinegroup}
In this section we give examples of nonzero functions in $L^2(G)$, where $G$ is the affine group, that have a linear dependency among some of its left translations. Let $\mathbb{R}$ denote the real numbers and let $\mathbb{R}^{\ast}$ be the set $\mathbb{R} \setminus \{ 0\}$. Recall that $\mathbb{R}$ is a group under addition and $\mathbb{R}^{\ast}$ is a group with respect to multiplication. The affine group, also known as the $ax + b$ group, is defined to be the semidirect product of $\mathbb{R}^{\ast}$ and $\mathbb{R}$. That is,
\[ G = \mathbb{R}^{\ast} \rtimes \mathbb{R}. \]
Let $(a, b)$ and $(c, d)$ be elements of $G$. The group operation on $G$ is given by $(a, b)(c, d) = (ac, b+ ad)$. The identity element of $G$ is $(1, 0)$ and $(a, b)^{-1} = (a^{-1}, -a^{-1}b)$. The left Haar measure on $G$ is $d\mu = \frac{dadb}{a^2}$ and the right Haar measure is $d\mu = \frac{dadb}{\vert a \vert}$. Thus $G$ is a nonunimodular group because the right and left Haar measures do not agree. So $f \in L^2(G)$ if and only if 
\[ \int_{\mathbb{R}} \int_{\mathbb{R}^{\ast}} \vert f(a, b) \vert^2 \frac{da db}{a^2} < \infty. \]
An irreducible unitary representation of $G$ can be defined on $L^2(\mathbb{R})$ by 
\[ \pi(a,b) f(x) = \vert a \vert^{-1/2} f\left(\frac{x-b}{a}\right), \]
where $(a, b) \in G$ and $f \in L^2(\mathbb{R})$. Before we show that $\pi$ is square integrable we recall some facts from Fourier analysis.

Let $f \in L^1(\mathbb{R}) \cap L^2(\mathbb{R})$, the Fourier transform of $f$ is defined to be 
\[ \hat{f}(\xi) = \int_{\mathbb{R}} f(x) e^{-2\pi i \xi x}\, dx, \]
where $\xi \in \mathbb{R}$. The Fourier transform can be extended to a unitary operator on $L^2(\mathbb{R})$. For $y \in \mathbb{R}$ we also have the following unitary operators on $L^2(\mathbb{R})$,
\[ T_y f(x) = f(x-y), E_yf(x) = e^{2\pi iyx} f(x) \]
\[ D_y f(x) = \vert y \vert^{-1/2} f\left(\frac{x}{y}\right), (y \neq 0). \]
Given $f, g \in L^2(\mathbb{R})$ the following relations are also true $\langle f, T_y g \rangle = \langle T_{-y}f, g \rangle, \langle f, E_y g\rangle = \langle E_{-y} f, g \rangle$ and $\langle f, D_y g \rangle = \langle D_{y^{-1}} f, g \rangle$. Furthermore, $\widehat{T_y f} = E_{-y} \widehat{f}$ and $\widehat{D_y f} = D_{y^{-1}} \widehat{f}$. Observe that for $(a, b) \in G$ and $f \in L^2(\mathbb{R})$,
\[ \pi(a, b) f(x) = T_b D_a f(x) = \vert a \vert^{-1/2} f \left( \frac{x-b}{a} \right). \]
Using the above relations it can be shown that for $f \in L^2(\mathbb{R})$,
\[ \int_G \langle f, \pi(a, b) f \rangle \, d\mu = \int_{\mathbb{R}} \int_{\mathbb{R}^{\ast}} \vert \langle f, T_b D_a f \rangle \vert^2 \frac{dadb}{a^2} = \Vert f \Vert^2_2 \int_{\mathbb{R}^{\ast}} \frac{ \vert \widehat{f} (\xi) \vert^2}{\vert \xi \vert} d\xi,\]
see \cite[Theorem 3.3.5]{HW89}. Thus $f \in L^2(\mathbb{R})$ is admissible if $\int_{\mathbb{R}^{\ast}} \frac{\vert \widehat{f}(\xi)\vert^2}{\vert \xi \vert} d\xi < \infty$. The function $f(x) = \sqrt{2\pi}xe^{-\pi x^2}$ satisfies this criterion since $\widehat{f}(\xi) = -\sqrt{2\pi} i \xi e^{-\pi\xi^2}$, which we obtained  by combining \cite[Proposition 2.2.5]{Pinskybook} with \cite[Example 2.2.7]{Pinskybook}. Hence $\pi$ is a square integrable, irreducible unitary representation of the affine group $G$. We are now ready to construct a nonzero function in $L^2(G)$ that has linearly dependent left translations. 

Let $\chi_{[0,1)}$ be the characteristic function on the interval $[0,1)$. It follows from the following {\em refinement equation}
\[ \chi_{[0,1)} (x) = \chi_{[0,1)} (2x) + \chi_{[0,1)} (2x-1) \]
that
\begin{equation}  \label{eq:dependexampleaffine}
\pi(1,0) \chi_{[0,1)} (x) = 2^{-1/2} \pi\left(2^{-1}, 0\right) \chi_{[0,1)} (x) + 2^{-1/2} \pi\left(2^{-1}, 2^{-1}\right) \chi_{[0,1)}(x).
\end{equation}
Thus $\chi_{[0,1)}$ has a linear dependency among the $\pi (a,b) \chi_{[0,1)}$, where $(a,b) \in G$. We now use $\chi_{[0,1)}$ to construct a nontrivial function in $L^2(G)$ that has linearly dependent left translations. Let $f\in L^2(\mathbb{R})$ be an admissible function for $\pi$ and let $(a, b) \in G$. Then the function
\[ F(a, b) = \langle \chi_{[0,1)}, \pi(a,b) f \rangle = \int_0^1 \vert a \vert^{-1/2} \overline{f \left( \frac{x-b}{a} \right)}\, dx \]
belongs to $L^2(G)$. By Proposition \ref{connectingprop}, $F(a,b)$ has linearly dependent left translations. More specifically,
\[ L(1,0) F(a,b) = 2^{-1/2} L(2^{-1}, 0) F(a,b) + 2^{-1/2} L (2^{-1}, 2^{-1}) F(a,b), \]
which translates to
\[ \int_0^1 \vert a \vert^{-1/2} \overline{f \left( \frac{x-b}{a} \right) } \, dx = \int_0^{1/2} \vert a \vert^{-1/2} \overline{ f \left( \frac{x-b}{a} \right)} \, dx + \int_{1/2}^1 \vert a \vert^{-1/2} \overline{f \left( \frac{x-b}{a} \right)} \, dx. \]

Equation \eqref{eq:dependexampleaffine} above was used in the proof
of \cite[Proposition 3.1]{Rosenblatt08} to show that the
time-frequency equation \eqref{eq:timefreqaffine} with $C = \sqrt{2}$
has a nonzero solution. Basically, equation \eqref{eq:timefreqaffine} is a reinterpretation of the above refinement equation where the representation $\pi$ is replaced by an equivalent representation. See \cite[Section 3]{Rosenblatt08} for the details.

We now turn our attention to the subgroup $K$ of the affine group $G$ which consists of all $(a, b) \in G$ for which $a >0$. This was the version of the affine group considered in \cite{Rosenblatt08}. The left Haar measure for $K$ is the same as the left Haar measure for $G$. Up to unitary equivalence there are two irreducible unitary infinite dimensional representations of $K$, see \cite[Section 6.7]{Follandbook95} for the details. One of these representations is given by
\[ \pi^+ (a,b) f(x) = a^{1/2} e^{2\pi i bx} f(ax) = E_b D_{a^{-1}} f(x), \]
where $(a,b) \in K$ and $f \in L^2(0, \infty)$. The representation
$\pi^+$ is square integrable. We are now ready to produce a
nontrivial function in $L^2(K)$ that has linearly dependent left
translations. From \eqref{eq:dependexampleaffine} we have
\[ \chi_{[0,1)} = 2^{-1/2} D_{2^{-1}} \chi_{[0,1)} + 2^{-1/2}T_{2^{-1}} D_{2^{-1}} \chi_{[0,1)}. \]
By taking Fourier transforms we obtain
\begin{align*}
\widehat{\chi}_{[0,1)} ( \xi) & = 2^{-1/2} D_2 \widehat{\chi}_{[0,1)} (\xi) + 2^{-1/2} E_{-2^{-1}}D_2 \widehat{\chi}_{[0,1)} (\xi) \\
                                          & = 2^{-1/2}\pi^+ (2^{-1},0)\widehat{\chi}_{[0,1)}(\xi) +2^{-1/2}\pi^+(2^{-1}, -2^{-1}) \widehat{\chi}_{[0,1)} (\xi).
\end{align*}
Hence, there is a linear dependency among the $\pi^+(a,b) \widehat{\chi}_{[0,1)}$, where $(a,b) \in K$. It follows from 
\[ \widehat{\chi}_{[0,1)} (\xi) = \frac{e^{-2\pi i \xi} - 1}{-2\pi i \xi} \]
that $\widehat{\chi}_{[0,1)} \in L^2(0, \infty)$. Pick an admissible function $f \in L^2(0, \infty)$ for $\pi^+$. Then the function 
\[ F(a,b) = \langle \widehat{\chi}_{[0,1)}, \pi^+(a,b) f \rangle = \int_0^{\infty} \widehat{\chi}_{[0,1)} (\xi) a^{1/2}e^{-2\pi ib\xi} \overline{f(\xi)} d\xi \]
is a member of $L^2(K)$. Proposition \ref{connectingprop} yields the following linear dependency in $L^2(K)$ among the left translations of $F(a,b)$.
\[ F(a,b) = 2^{-1/2} L(2^{-1}, 0) F(a,b) + 2^{-1/2} L(2^{-1}, -2^{-1}) F(a,b). \]
This equation can easily be verified by using the relations
\[ \widehat{\chi}_{[0,1)} ( \frac{\xi}{2})( 1 + e^{-\pi i \xi}) = \frac{(e^{- \pi i \xi} - 1)(1 + e^{- \pi i \xi})}{-2\pi i \xi} = \widehat{\chi}_{[0,1)} (\xi). \]

\section{Discrete groups and the Atiyah conjecture}\label{discretegroupzerodiv}
In this section we connect the problem of linear independence of left translations of a function to the Atiyah conjecture. Unless otherwise stated we make the assumption that all groups in this section are discrete. For discrete groups the Haar measure is counting measure. Let $f$ be a complex-valued function on a group $G$. We will represent $f$ as a formal sum $\sum_{g \in G} a_g g$, where $a_g \in \mathbb{C}$ and $f(g) = a_g$. Denote by $\ell^2(G)$ those formal sums for which $\sum_{g \in G} \vert a_g \vert^2 < \infty$, and $\mathbb{C}G$, the group ring of $G$ over $\mathbb{C}$ will consist of all formal sums that satisfy $a_g = 0$ for all but finitely many $g$. The group ring $\mathbb{C}G$ can also be thought of as the set of all functions on $G$ with compact support and $\ell^2(G)$ is a Hilbert space with Hilbert basis $\{ g \mid g \in G\}$. If $g \in G$ and $f = \sum_{x \in G} a_x x \in \ell^2(G)$, then the left translation of $f$ by $g$ is represented by the formal sum $\sum_{x \in G} a_{g^{-1}x} x$ since $L(g) f(x) = f(g^{-1}x)$. Suppose $\alpha = \sum_{g \in G} a_g g \in \mathbb{C}G$ and $f = \sum_{g \in G} b_g g \in \ell^2(G)$. We define a multiplication, known as convolution, $\mathbb{C}G \times \ell^2(G) \rightarrow \ell^2(G)$ by
\[ \alpha \ast f = \sum_{g,h \in G} a_g b_h gh = \sum_{g \in G} \left( \sum_{h \in G} a_{gh^{-1}} b_h \right) g. \]
Sometimes we will write $\alpha f$ instead of $\alpha \ast f$. Left multiplication by an element of $\mathbb{C}G$ is a bounded
linear operator on $\ell^2(G)$. So $\mathbb{C}G$ can be considered as
a subring of $\mathcal{B}(\ell^2(G))$, the space of bounded linear
operators on $\ell^2(G)$. 
We shall say that $G$ is torsion-free if the only element of finite order in $G$ is the identity element of $G$.
The strong Atiyah conjecture for the group $G$ is concerned with the
values the $L^2$-Betti numbers can take, and it implies the
following conjecture, which can also be considered an analytic version
of the zero divisor conjecture.
\begin{Conj}\label{zerodivconj}
Let $G$ be a torsion-free group.  If $0 \neq \alpha \in \mathbb{C}G$ and $0 \neq f \in \ell^2(G)$, then $\alpha \ast f \neq 0$.
\end{Conj}
The hypothesis that $G$ is torsion-free is essential. Indeed, let 1 be the identity element of $G$ and let $g \in G$ such that $g \neq 1$ and $g^n =1$ for some $n \in \mathbb{N}$. Then $(1 + g + \cdots + g^{n-1}) \ast (1 - g) = 0$. The Atiyah conjecture is important in the study of von Neumann dimension. For further information see \cite{Cohen81, Linnell91,Linnell92} and \cite[Section 10]{Lueck02}.
In particular, the assertion of Conjecture \ref{zerodivconj} is known
for free groups, left-ordered groups and elementary amenable groups.

The following proposition gives the link between zero divisors and the linear independence of left translations of a function. 
\begin{Prop} \label{connecttozerodiv}
Let $G$ be a discrete group and let $f \in \ell^2(G)$.  Then $f$ has linearly independent left translations if and only if $\alpha \ast f \neq 0$ for all nonzero $\alpha \in \mathbb{C}G$.
\end{Prop}
\begin{proof}
Let $g \in G$ and let $f = \sum_{x \in G} a_x x \in \ell^2(G)$. Then 
\[ g \ast f = \sum_{x \in G} a_x gx = \sum_{x \in G} a_{g^{-1}x} x = L(g)f. \]
Consequently, if $g_1, \dots, g_n \in G$ are distinct and $c_1, \dots, c_n$ are constants, then
\[ \sum_{k =1}^n c_k L(g_k) f = \sum_{k=1}^n c_k g_k \ast f = \left( \sum_{k=1}^n c_k g_k \right) \ast f. \]
The proposition now follows since $\sum_{k=1}^n c_k g_k \in \mathbb{C}G$.
\end{proof}
As we saw in Section \ref{affinegroup} there are nontrivial, square integrable functions on the affine group that have a linear dependency among their left translations. Since all nonidentity elements of the affine group have infinite order, it seems reasonable by taking a discrete subgroup $D$ of the affine group, such as $1 \rtimes \mathbb{Z}$, we might be able to construct a nontrivial function in $\ell^2(D)$ that has a linear dependency among its left translations. It would then be an immediate consequence of Proposition \ref{connecttozerodiv} that Conjecture \ref{zerodivconj} is false. However, for discrete subgroups $D$ of the affine group it is not true that there exists a nonzero function in $\ell^2(D)$ with a linear dependency among its left translations. Indeed, the affine group is a solvable Lie group, and all discrete subgroups of solvable Lie groups are polycyclic. By \cite[Theorem 2]{Linnell91} Conjecture \ref{zerodivconj} is true for torsion-free elementary amenable groups, a class of groups that contain all torsion-free polycyclic groups.

\section{Proof of Theorem \ref{discretesubgroup}}\label{proofofdiscretesubthm}
In this section we prove Theorem \ref{discretesubgroup}.  Recall that
our standing assumptions on the group $G$ is that it is locally
compact, Hausdorff with left invariant Haar measure $\mu$. Let $g_1,
\dots, g_n$ be elements of $G$ and let $c_1, \dots, c_n \in
\mathbb{C}$ be some constants. Set $\theta = \sum_{k=1}^n c_k L(g_k)$. So $\theta \in \mathcal{B}(L^2(G))$, the set of bounded linear operators on $L^2(G)$. Define
\[ \mathbb{C}G = \{ \sum_{g \in G} a_g L(g)\mid a_g = 0 \text{ for all but finitely many } g \in G\}. \]
Note that there exists a nonzero $f \in L^2(G)$ with linearly dependent left translations if and only if there exists a nonzero $\theta \in \mathbb{C}G$ with $\theta f = 0$. 

For the rest of this section $H$ will denote a discrete subgroup of $G$. We will also assume that $G$ is $\sigma$-compact in addition to our standing assumptions on $G$. The subgroup $H$ acts on $G$ by left multiplication. By \cite[Proposition B.2.4]{BHV08} there exists a Borel fundamental domain for this action of $H$ on $G$. More precisely, there exists a Borel subset $B$ of $G$ such that $hB \cap B = \emptyset$ for all $h \in H \setminus 1$ and $G = HB$ (thus $B$ is a system of right coset representatives of $H$ in $G$ which is also a Borel subset). If $X$ is a Borel subset of $G$, then we will identify $L^2(X)$ with the subspace of $L^2(G)$ consisting of all functions on $G$ whose support is contained in $X$.

Let $\{q_i \mid i\in \mathcal{I}\}$ be a
Hilbert basis for $L^2(B)$.  We claim that $S := \{L(h)q_i
\mid h\in H,\ i \in \mathcal{I}\}$ is a Hilbert basis for $L^2(G)$.
First we show that $S$ is orthonormal.  Write $hq_i$ for $L(h)q_i$.
If $h\ne k$, then $\langle
hq_i,kq_j\rangle = 0$, because the supports of $hq_i$ and $hq_j$ are
contained in $hB$ and $kB$ respectively, which are disjoint subsets.
On the other hand if $h = k$, then $\langle hq_i,hq_j \rangle
= \langle q_i,q_j\rangle$, because the Haar measure is left
invariant.  This proves that $S$ is orthonormal.  Finally we show
that the closure of the linear span $\overline{S}$ of $S$ is
$L^2(G)$.  Denote by $\chi_{hB}$ the characteristic function on $hB$. If $f \in L^2(G)$, then we may write $f
=\sum_{h\in H} f_h$, where $f_h = \chi_{hB} f$ (so $f_h$ has support
contained in $hB$).  Thus it will be sufficient to show that $L^2(hB)
\subseteq \overline{S}$.  Since $\overline{S}$ is invariant under
$H$, it will be sufficient to show that $L^2(B) \subseteq
\overline{S}$, which is obvious because the $q_i$ form a Hilbert
basis for $L^2(B)$.

For $i \in \mathcal{I}$,
let $S_i = \{ L(h)q_i \mid h \in H\}$ and let $\overline{S}_i$ denote
the closure of the linear span of $S_i$.
Now $L^2(G) = \bigoplus_{i \in \mathcal{I}}
\overline{S}_i$ (where $\bigoplus$ indicates the Hilbert direct sum).
The spaces $\overline{S}_i$ are isometric to $\ell^2(H)$. Indeed,
define a map $T_i$ from the Hilbert basis $S_i$
of $\overline{S}_i$ to the Hilbert basis $H$ of $\ell^2(H)$ via
$L(h)q_i \mapsto h$. Extend $T_i$ linearly to obtain an isometry 
$T_i \colon \overline{S}_i \rightarrow \ell^2(H)$. Moreover, the
isometry $T_i$ intertwines the natural left actions of $H$ on
$\overline{S}_i$ and $\ell^2(H)$.  Also let $\pi_i$ denote the
projection of $L^2(G)$ onto $\overline{S}_i$.  Then $\pi_i$ also
intertwines the natural left actions of $H$ on $L^2(G)$ and
$\overline{S}_i$.
Now suppose that there exists a nonzero $f \in L^2(G)$ and a nonzero $\theta \in \mathbb{C}H$ that
satisfies $\theta f =0$.  Then $k := T_i\pi_i f \ne 0$ for some $i$, and
$\theta * k =\theta k =0$ because $T_i\pi_i$ commutes with $\mathbb{C}H$.
Furthermore, $k \in l^2(H)$.
We can summarize the above as follows:
\begin{Prop} \label{Pzerodivisors}
Let $H$ be a discrete subgroup of the $\sigma$-compact locally
compact group $G$ and let $\theta \in \mathbb{C}H$.  If $\theta f =
0$ for some nonzero $f \in L^2(G)$, then $\theta *k = 0$ for some
nonzero $k \in \ell^2(H)$.
\end{Prop}

Now let $H$ be a torsion-free group which satisfies the strong Atiyah
conjecture, e.g.\ a torsion-free elementary amenable group.
Then for $0 \ne \theta \in \mathbb{C}H$, we know that $\theta *k \ne
0$ for all non-zero $k \in \ell^2(H)$. It follows from Proposition
\ref{Pzerodivisors} that $\theta  f \ne 0$ for all nonzero $f \in
L^2(G)$,
in other words, any nonzero element of $L^2(G)$ has linearly independent $H$-translations. The proof of Theorem \ref{discretesubgroup} is now complete.

In a similar fashion, we can prove
\begin{Thm} \label{discretesubgroup1}
Let $G$ be a locally compact $\sigma$-compact group and let $H$ be an
amenable discrete subgroup of $G$.  If $\alpha$ is a non-zerodivisor
in $\mathbb{C}H$, then $\alpha * f \ne 0$ for all nonzero $f \in
L^2(G)$.
\end{Thm}
\begin{proof}
Since $\alpha \beta \ne 0$ for all nonzero $\beta \in \mathbb{C}H$,
it follows that $\alpha \beta \ne 0$ for all nonzero $\beta \in
\ell^2(H)$ by \cite[Theorem]{Elek03} (or use
\cite[Theorem 6.37]{Lueck02}).  The result now follows from
Proposition \ref{Pzerodivisors}.
\end{proof}

We saw in Section \ref{affinegroup} that for the affine group $A$ there exist nonzero $f$ in $L^2(A)$ with linearly dependent left translations. However, $\mathbb{Z}$ can be identified with the discrete subgroup $1 \rtimes \mathbb{Z}$ of $A$. A direct consequence of Theorem \ref{discretesubgroup} is 
\begin{Cor} \label{affinecor}
Let $A$ be the affine group. Then every nonzero $f$ in $L^2(A)$ has linearly independent left $\mathbb{Z}$-translations.
\end{Cor}

As noted in Section \ref{discretegroupzerodiv}, if $H$ is a discrete
group, we may regard
$\mathbb{C}H$ as a subalgebra of $\mathcal{B}(\ell^2(H))$.
Recall that the reduced group $C^*$-algebra of $H$, denoted
$\Cr(H)$, is the operator norm closure of
$\mathbb{C}H$ in $\mathcal{B}(\ell^2(H))$, and the group von Neumann
algebra of $H$, denoted $\mathcal{N}(H)$, is the weak closure of
$\mathbb{C}H$ in $\mathcal{B}(\ell^2(H))$.
We can also identify the norm and weak closures of $\mathbb{C}H$ in
$\mathcal{B}(L^2(G))$ with $\Cr(H)$ and $\mathcal{N}(H)$
respectively.  Though this is not needed in the sequel, we
hope it maybe useful to record this.

For $\theta \in \mathcal{B}(L^2(G))$ or $\mathcal{B}(\ell^2(H))$, let
$\Vert \theta \Vert$ or $\Vert \theta \Vert'$ denote the
corresponding operator norms respectively.  We retain the notation
used in the proof of Proposition \ref{Pzerodivisors}.
Observe that we have a natural isomorphism
$\mathcal{B}(\overline{S}_i) \rightarrow \mathcal{B}(\ell^2(H))$
induced by $T_i$.  Furthermore $L^2(G) = \bigoplus_{i \in \mathcal{I}}
\overline{S}_i$ (where $\bigoplus$ indicates the Hilbert direct sum),
and this a decomposition as left $\mathbb{C}H$-modules.
We will need the following:
\begin{Lem} \label{opnormeq} Let $\theta \in \mathbb{C}H$. Then $\Vert \theta \Vert = \Vert \theta \Vert'$
\end{Lem}
\begin{proof}
Note that $\theta$ can be considered as an operator on $L^2(G)$ or
$\ell^2(H)$.  If $u \in L^2(G)$, we may write $u = \sum_{i \in \mathcal{I}} u_i$ with $u_i \in \overline{S}_i$, so
\begin{align*}
\|\theta\| = \sup_{u \in L^2(G),\ \|u\|_2 = 1} \|\theta u\|_2 &=
\sup_{u \in L^2(G),\ \|u\|_2 = 1} \|\theta \sum_{i\in \mathcal{I}}
u_i\|_2\\
&\le \sup_{u \in L^2(G),\ \|u\|_2 = 1} \sqrt{\sum_{i \in \mathcal{I}}\|\theta\|^{\prime \,
2} \|u_i\|_2^2 } = \|\theta\|'.
\end{align*}
Fix $\iota \in \mathcal{I}$.  Then
\[
\|\theta\|' = \sup_{u \in \overline{S}_\iota,\ \|u\|_2 = 1}
\|\theta u\|_2 \le
\sup_{u \in L^2(G),\ \|u\|_2 = 1} \|\theta u\|_2 \le \|\theta\|.
\]
Therefore, $\Vert \theta \Vert = \Vert \theta \Vert'$.
\end{proof} 

Denote by $\mathcal{O}(H)$ the operator norm closure, and
$\mathcal{W}(H)$ the weak closure of $\mathbb{C}H$ in
$\mathcal{B}(L^2(G))$. The space $\mathcal{W}(H)$ is a von Neumann
algebra and by the double commutant theorem is equal to the strong
closure of $\mathbb{C}H$ in $\mathcal{B}(L^2(G))$. Note that $\mathcal{O}(H) \subseteq \mathcal{W}(H)$ and $\Cr(H) \subseteq \mathcal{N}(H)$. We now relate these various algebras:

\begin{Prop} \label{cstariso}
There is a $\ast$-isomorphism $\alpha \colon \mathcal{W}(H)
\rightarrow \mathcal{N}(H)$. Moreover, $\alpha$ preserves the
operator norm and maps $\mathcal{O}(H)$ onto $\Cr(H)$.
\end{Prop}
\begin{proof}
Recall that for $u \in L^2(G)$, we can uniquely write $u = \sum_{i \in
\mathcal{I}} u_i$ with $u_i \in \overline{S}_i$. Let $\theta \in
\mathcal{W}(H)$. Then there exists a net $(\theta_i)$ in
$\mathbb{C}H$ which converges strongly to $\theta$. Therefore for
every $u \in L^2(G)$, the net $(\theta_iu)$ is convergent in
$L^2(G)$, consequently the net $(\theta_i u_j)$ is convergent for
every $j$, in particular $(\theta_i f)$ is a Cauchy net in
$\ell^2(H)$ for every $f \in \ell^2(H)$. We deduce that $(\theta_i)$
is a Cauchy net in $\mathcal{B}(\ell^2(H))$ (in the strong operator
topology) and hence converges to an operator $\theta^{\prime} \in
\mathcal{N}(H)$. We note that $\theta^{\prime}$ doesn't depend on the
choice of the net $(\theta_i)$ and therefore we have a well-defined map $\alpha \colon \mathcal{W}(H) \rightarrow \mathcal{N}(H)$, where $\alpha (\theta) = \theta^{\prime}$ and $\alpha$ is the identity on $\mathbb{C}H$.

We now construct the inverse to $\alpha$ by reversing the above steps. Let $\phi \in \mathcal{N}(H)$. By the Kaplansky density theorem there exists a net $(\theta_i)$ in $\mathbb{C}H$ which converges strongly to $\phi$ and $\Vert \theta_i \Vert^{\prime}$ bounded. Thus $\Vert \theta_i \Vert$ is bounded because $\Vert \theta_i \Vert = \Vert \theta_i \Vert^{\prime}$ for each $i$ by Lemma \ref{opnormeq}. Now let $u \in L^2(G)$. If $\mathcal{J}$ is a finite subset of $\mathcal{I}$, set $v_{\mathcal{J}} = \sum_{j \in \mathcal{J}} u_j$. Then $(\theta_i v_j)$ converges in $L^2(G)$ for every $\mathcal{J}$. Since $\Vert \theta_i \Vert$ is bounded, it follows that $(\theta_i u )$ is convergent in $L^2(G)$ and we conclude that $(\theta_i)$ converges strongly to an operator $\tilde{\phi} \in \mathcal{B}(L^2(G))$. It follows that we have a well-defined map $\phi \rightarrow \tilde{\phi} \colon \mathcal{N}(H) \rightarrow \mathcal{W}(H)$, which is the inverse to $\alpha$. 

It is easily checked that $\alpha$ is a $\ast$-isomorphism and
therefore is an isomorphism of $C^*$-algebras, in particular it
preserves the operator norm. We deduce that $\alpha$ maps
$\mathcal{O}(H)$ onto $\Cr(H)$.
\end{proof}

\begin{Rem}
Proposition \ref{cstariso} can be used to give a different
proof of Proposition \ref{Pzerodivisors}; for details, see
\cite[Chapter 2.5]{Ahmedthesis}.
\end{Rem}

\section{The Weyl-Heisenberg group} \label{WeylHeisenberggroup}
In this section we use techniques developed in this paper to
determine when $f =0$ is the only solution to the time-frequency
equation \eqref{eq:timefreqheisen}. The relevant group here is the Weyl-Heisenberg group since it has an irreducible representation that is square integrable. 

Let $n \in \mathbb{N}$. The {\em Heisenberg group} $H_n$ is the set of $(n+2) \times (n+2)$ matrices of the form
\[ \begin{pmatrix} 1  &  a  &  z  \\
                                0   &  1_n  & b \\
                                0   &   0     &  1  \end{pmatrix} \]
where $a$ is a $1 \times n$ matrix, $b$ is a $n \times 1$ matrix, the
zero in the $(2,1)$ position is the $n \times 1$ zero matrix, the
zero in the $(3,2)$ position is the $1 \times n$ zero matrix, and the
$1_n$ in the $(2,2)$ position is the $n \times n$ identity matrix.
Another way to represent $H_n$ is as the product $\mathbb{R} \times
\widehat{\mathbb{R}^n} \times \mathbb{R}^n$.  Here we view
$\mathbb{R}^n$ as $n\times 1$ column matrices and
$\widehat{\mathbb{R}^n}$ as $1 \times n$ row matrices.
For $(z_1, a_1, b_1), (z_2, a_2, b_2) \in H_n$ the group law becomes $(z_1, a_1, b_1) (z_2, a_2, b_2) = (z_1 + z_2 +a_1 \cdot b_2, a_1 + a_2, b_1 + b_2)$. Thus the identity element in $H_n$ is $(0,0,0)$ and $(z, a, b)^{-1} = (a \cdot b -z, -a, -b)$. For $f \in L^2(\mathbb{R}^n)$ and $(z, a, b) \in H_n$ define
\[ \pi(z, a, b)f(x) = e^{2\pi iz}e^{-2\pi i a \cdot b} e^{2\pi i a \cdot x} f(x - b). \]
It turns out that $\pi$ is a representation of $H_n$ on
$L^2(\mathbb{R}^n)$. Indeed, let $(z_1, a_1, b_1)$,
$(z_2, a_2, b_2) \in H_n$. Then
\begin{align*}
\pi(z_1, a_1, b_1) &\bigl( \pi(z_2, a_2, b_2) f(x) \bigr) = \pi(z_1, a_1, b_1) \bigl( e^{2 \pi i z_2} e^{-2\pi i a_2 \cdot b_2} e^{2 \pi i a_2 \cdot x} f(x - b_2) \bigr)  \\
  & = e^{2\pi i z_1} e^{2 \pi i z_2} e^{-2\pi i a_1 \cdot b_1} e^{-2\pi i a_2 \cdot b_2} e^{2 \pi i a_1 \cdot x} e^{2 \pi i a_2 \cdot (x - b_1)} f(x - b_2 - b_1) \\
  & = e^{2\pi i (z_1 + z_2)} e^{-2\pi i (a_1 \cdot b_1 +a_2 \cdot b_2)} e^{-2\pi i a_2 \cdot b_1} e^{2\pi i a_2\cdot x} f(x - (b_1 + b_2)) \\
  & = e^{2\pi i (z_1 + z_2 + a_1 \cdot b_2)} e^{-2\pi i (a_1 + a_2) \cdot (b_1 + b_2)} e^{2\pi i (a_1 + a_2)\cdot x} f(x - (b_1 + b_2))\\
  & = \bigl( \pi(z_1, a_1, b_1) \pi(z_2, a_2, b_2)\bigr)  f(x).
\end{align*}
Let $Z = \langle (2\pi, 0, 0) \rangle$, the subgroup of $H_n$
generated by $(2\pi, 0, 0)$. Set $\tilde{H}_n = H_n/Z$. The group
$\tilde{H}_n$ is known as the {\em Weyl-Heisenberg group}. Clearly $Z
= \ker \pi$ and so $\pi$ induces a representation $\tilde{\pi}$ on
$\tilde{H}_n$. Observe that $\tilde{H}_n = \{ (t, a, b) \mid t \in
\mathbb{T}, a, b \in \mathbb{R}^n \}$ (here $\mathbb{T}$ is the unit
circle $\{z\in \mathbb{Z} \mid |z|=1\}$).
The Lebesgue measure on $H_n =
\mathbb{R} \times \widehat{\mathbb{R}^n} \times \mathbb{R}^n$ is left and right
invariant Haar measure on $H_n$. Similarly, Lebesgue measure on
$\mathbb{T} \times \widehat{\mathbb{R}^n} \times \mathbb{R}^n$ is left and
right invariant Haar measure on $\tilde{H}_n$ (here Lebesgue measure
on $\mathbb{T}$ is normalized so that $\int_{\mathbb{T}} dt =1$). The
next result was proved in \cite[Proposition 3.2.4]{HW89} for the special case $n=1$. By interchanging the roles of $a$ and $b$ the proof given there carries through verbatim to our case.
\begin{Prop} \label{sqintWeylHeis} If $f, g \in L^2(\mathbb{R}^n)$, then 
\[ \int_{\widehat{\mathbb{R}^n}} \int_{\mathbb{R}^n} \int_{\mathbb{T}} \vert \langle f, \tilde{\pi} (t, a, b) g \rangle \vert^2 \, dt db da = \Vert f \Vert_2^2 \Vert g \Vert_2^2. \]
\end{Prop}
\begin{Cor} \label{irredWeylHeis} The representation $\tilde{\pi}$ of
$\tilde{H}_n$ on $L^2(\mathbb{R}^n)$ is irreducible and square integrable, and every $g \in L^2(\mathbb{R}^n)$ is admissible.
\end{Cor}
\begin{proof}
By taking $f = g$ in the above proposition we see immediately that every element of $L^2(\mathbb{R}^n)$ is admissible. Suppose $g \in L^2(\mathbb{R}^n) \setminus \{ 0 \}$ is fixed and assume $f \in L^2(\mathbb{R}^n)$ satisfies $\langle f, \tilde{\pi}(t, a, b) g \rangle = 0$ for all $(t, a, b) \in \tilde{H}_n$. Then $\Vert f \Vert_2 \Vert g \Vert_2 =0$ and it follows that $f=0$. Hence $\tilde{\pi}$ is irreducible as desired.
\end{proof}

\begin{Prop} \label{timefreq}
Let $n\in \mathbb{N}$, let $(a_k,b_k) \in \mathbb{R}^{2n}$
be distinct nonzero elements such that 
$(a_k,b_k)$ generate a discrete subgroup of $\mathbb{R}^{2n}$ and
$a_h\cdot b_k \in \mathbb{Q}$ for all $h,k$ (where $k,h \in
\mathbb{N}$).  If $r\in \mathbb{N}$ and
\[\sum_{k=1}^{r} c_ke^{2\pi i b_k \cdot t}f(t+a_k)=0\]
with $0\ne c_k\in \mathbb{C}$ constants, then $f=0$. 
\end{Prop}
\begin{proof}
We have $\mathbb{R}^{2n} = \tilde{H}_n/\mathbb{T}$.  Lift the
$(a_k,b_k)$ to the elements $g_k := (1,a_k,b_k)
\in \tilde{H}_n$.  Note that
the hypothesis $a_h\cdot b_k \in \mathbb{Q}$ ensures that $\langle
g_1, \dots, g_r\rangle$ is a discrete subgroup of $\tilde{H}_n$.
We claim that if $0 \ne d_k \in \mathbb{C}$, then $\alpha
:= \sum_{k=1}^r d_k g_k$ is a non-zerodivisor in $\mathbb{C}\tilde{H}_n$.
Indeed if $0 \ne \beta \in \mathbb{C}\tilde{H}_n$ and $\alpha\beta =
0$, let $T$ be a transversal for $\mathbb{T}$ in $\tilde{H}_n$
containing $\{g_1,\dots,g_r\}$
and write $\beta = \sum_{t\in T} \beta_t t$ where
$\beta_t \in \mathbb{C}\mathbb{T}$.  Since $\mathbb{R}^{2n}$ is an
ordered group, we can apply a leading term argument: let 
$k$ be such that $g_k\in T$ is largest and let
$s \in T$ be the largest element such that $\beta_s \ne 0$.  Then
by considering $g_ks$, we see that $\alpha\beta \ne 0$ because
$d_k\beta_s \ne 0$, which is a contradiction.
The result now follows
from Proposition \ref{connectingprop}, Corollary \ref{irredWeylHeis}
and Theorem \ref{discretesubgroup1}.
\end{proof}

\section{Shearlet Groups}\label{shearletgroups}

We now investigate the problem of linear independence of left translations of functions in $L^2(S)$, where $S$ denotes the shearlet group. This fits the theme of our paper since $S$ has an irreducible, square integrable representation on $L^2(\mathbb{R}^2)$. We begin by defining the shearlet group.

For $a \in \mathbb{R}^+$ (the positive real numbers)
and $s \in \mathbb{R}$ let $A_a =
\begin{pmatrix}  a   &  0  \\ 0   &  \sqrt{a} \end{pmatrix}$,
$S_s = \begin{pmatrix}
1  &  s \\
0  & 1 \end{pmatrix}$ and let $G = \{ S_s A_a \mid a\in \mathbb{R}^+,
s \in \mathbb{R}\}$. The {\em shearlet group} $S$ is defined to be $S
= G \ltimes \mathbb{R}^2$.  The group multiplication for $S$ is given
by $(M,t)(M',t') = (MM', t + Mt')$, where $M \in G$ and $t \in
\mathbb{R}^2$ (here we are considering elements of $\mathbb{R}^2$ as
column vectors).
The left Haar measure for $S$ is $\frac{da ds dt}{a^3}$ and the right Haar measure for $S$ is $\frac{da ds dt}{a}$, so $S$ is a nonunimodular group. A representation $\pi$ of $S$ on $L^2(\mathbb{R}^2)$ can be defined by
\[ \pi(S_s A_a, t)f(x) = a^{-3/4} f( (S_sA_a)^{-1} (x-t)). \]

The representation $\pi$ is square integrable and irreducible, see
\cite[\S 2]{DahlkeKut08} for the details. We shall write $f_{ast}$ to indicate $\pi(S_sA_a, t) f$. The function $f_{ast}$ is also known as the {\em shearlet transform} of $f$. Since the shearlet transform is realized by an irreducible, square integrable representation of $S$ on $L^2(\mathbb{R}^2)$, the question of linear independence of the left translations of a function in $L^2(S)$ is related to the question of the linear independence of the shearlets of a function in $L^2(\mathbb{R}^2)$. The question of linear independence of the shearlet transforms of $f$ now becomes: Is $f=0$ the only solution in $L^2(\mathbb{R}^2)$ that satisfies 
\begin{equation}\label{eq:independentshearlets}
 \sum_{k=1}^n c_kf_{a_ks_kt_k} = 0 
\end{equation}
where $c_k$ are nonzero constants and $(a_k, s_k, t_k) \in \mathbb{R}^+ \times \mathbb{R} \times \mathbb{R}^2$? 

\begin{Prop}\label{shearletdependency}
Let $S$ be the shearlet group. There exists a nonzero function in $L^2(S)$ that has linearly dependent left translations.
\end{Prop}
\begin{proof}
The proposition will follow immediately from Proposition \ref{connectingprop} if we can show there exists a nonzero $f \in L^2(\mathbb{R}^2)$ that satisfies \eqref{eq:independentshearlets}, which we now do. Combining \cite[Theorem 4.6]{KutyniokSauer09} and \cite[Example 5]{Grohs13} we see that there exists a continuous nonzero $f \in L^2(\mathbb{R}^2)$ that satisfies
\begin{equation} \label{eq:shearletrefine}
f(x) = \sum_{\beta \in \mathbb{Z}^2} a(\beta) f(A^{-1}_4 x - \beta),
\end{equation}
where $a(\beta) \in \mathbb{C}\mathbb{Z}^2$ and $x \in \mathbb{R}^2$. The function $f$ is said to be {\em refinable}. In the literature $a(\beta)$ is often referred to as a mask. The important thing here is that $a(\beta)$ has finite support. If $\beta = \left( \begin{array}{c} b_1 \\
                                  b_2 \end{array} \right)$, then set $\beta' = \left( \begin{array}{c} 4b_1 \\
                                                                                                                                        2b_2  \end{array} \right)$. Using \eqref{eq:shearletrefine} we obtain
\begin{align*}
\pi(S_0A_1,0)f(x) & = \sum_{\beta \in \mathbb{Z}^2} a(\beta) 4^{3/4} 4^{-3/4} f(A^{-1}_4 x - \beta) \\
                            & = \sum_{\beta \in \mathbb{Z}^2} a(\beta) 4^{3/4} 4^{-3/4} f(A^{-1}_4 x - A^{-1}_4 \beta') \\
                            & = \sum_{\beta \in \mathbb{Z}^2} 4^{3/4} a(\beta) \pi(S_0 A_4, \beta') f(x) \\
                            & = \sum_{\beta \in \mathbb{Z}^2} 4^{3/4} a(\beta) f_{4, 0, \beta'}(x).
\end{align*}
Hence, there is a linear dependency among the shearlet transforms of $f$, proving the proposition.
\end{proof}

\begin{Rem}
The refinable function $f$ used in the proof of the previous proposition has compact support since $a(\beta)$ has finite support \cite[Theorem 5]{Grohs13}. Compare this to \cite[Theorem 4.3]{MaPetersen} where it was shown, in a slightly different setting, that a compactly supported separable shearlet system is linearly independent. Thus it appears that in general the hypothesis of separability is important.
\end{Rem}

The next result gives a sufficient condition for linear independence of a shearlet system. 

\begin{Prop} \label{shearletindep} 
Let $0 \neq f \in L^2(\mathbb{R}^2)$. Then $\{ f_{1nt} \mid n \in \mathbb{Z}, t \in \mathbb{Z}^2 \}$ is a linearly independent set.
\end{Prop}

\begin{proof}
Let $H = \{ S_nA_1 \mid n \in \mathbb{Z}\}$, then $K = H \ltimes
\mathbb{Z}^2$ is a torsion-free discrete subgroup of $S$. Because $H$
and $\mathbb{Z}^2$ are solvable, $K$ is solvable and thus satisfies
the strong Atiyah conjecture. By Theorem \ref{discretesubgroup} the $K$-left
translations of a function in $L^2(S)$ are linearly independent.
The proposition now follows from Proposition \ref{connectingprop}.
\end{proof}

The results obtained in this section are similar to the results from Section \ref{affinegroup} for the affine group. This is not surprising since the shearlet transform involves a dilation and a translation.

\section{Virtually abelian groups}

In this section we consider virtually abelian groups, that is
groups with an abelian subgroup of finite index.

\begin{Prop} \label{Pabelianbyfinite}
Let $G$ be a locally compact group which has an abelian
closed subgroup $A$ of finite index, and let $1 \le p \in
\mathbb{R}$.  Assume that if $0 \ne \phi \in \mathbb{C}A$ and $0 \ne
f \in L^p(A)$, then $\phi f \ne 0$.  Let $0 \ne f \in L^p(G)$, let $H
\leqslant G$ and let $\theta \in \mathbb{C}H$.
\begin{enumerate}[\normalfont(a)]
\item
If $\theta$ is a nonzero divisor in $\mathbb{C}H$, then $\theta f \ne
0$.

\item
If $H$ is torsion free and $\theta \ne 0$, then $\theta f \ne 0$.
\end{enumerate}
\end{Prop}

\begin{proof}
Note that $\mathbb{C}A$ is an integral domain.
Let $B$ be the intersection of the
conjugates of $A$ in $G$, then $B$ is a closed abelian normal subgroup of finite
index in $G$.  Let $\{a_1, \dots, a_m\}$ be a set of coset
representatives for $B$ in $A$.  Then $L^p(A) = \bigoplus_{i=1}^m
L^p(B)a_i$ and we see that if $0 \ne \phi \in \mathbb{C}B$ and $0 \ne
f \in L^p(B)$, then $\phi f \ne 0$.
Let $\{g_1,
\dots, g_n\}$ be a set of coset representatives for $B$ in $G$.  Then
$L^p(G) = \bigoplus_{i=1}^n L^p(B)g_i$.  We may view this as an
isomorphism of $\mathbb{C}B$-modules.  Set $S = \mathbb{C}B
\setminus\{0\}$.  Then we may form the ring
of fractions $S^{-1} \mathbb{C}G$.
Since every element of $S$ is a non-zerodivisor in $\mathbb{C}G$, it
follows that $S^{-1}\mathbb{C}G$ is a ring containing $\mathbb{C}G$.
Furthermore $S^{-1}\mathbb{C}B$ is a field, and $S^{-1}\mathbb{C}G$
has dimension $n$ over this field.  Therefore $S^{-1}\mathbb{C}G$ is
an artinian ring, and since $S^{-1}\mathbb{C}B$ is a field of
characteristic zero, we see that $S^{-1}\mathbb{C}G$
is a semisimple artinian ring, by Maschke's theorem.  We deduce
that non-zerodivisors in $S^{-1}\mathbb{C}G$ are invertible.
Using \cite[Theorem 10.8]{GoodearlWarfield04}, we may form the
$S^{-1}\mathbb{C}G$-module $S^{-1}L^p(G)$.

\begin{enumerate}[\normalfont(a)]
\item
If $\theta$ is non-zerodivisor in $\mathbb{C}H$,
then $\theta$ is a non-zerodivisor in $\mathbb{C}G$ and hence is
invertible in $S^{-1}\mathbb{C}G$, so $\theta^{-1}$ exists.  We may
regard $f$ as an element of $S^{-1}L^p(G)$, because
$S^{-1}L^p(G)$ contains $L^p(G)$.  So if $\theta f = 0$, then
$\theta^{-1}\theta f = 0$, consequently $f = 0$ and we have a
contradiction.

\item
If $H$ is torsion free, then we know that every non-zero element
of $\mathbb{C}H$ is a non-zerodivisor in $\mathbb{C}H$;
this was first proved by K.~A.~Brown \cite{Brown76}.
Thus the result follows from (a).\qedhere
\end{enumerate}
\end{proof}

We now use the previous result to give the following generalization of \cite[Theorem 1.2]{EdgarRosenblatt79}.

\begin{Thm}\label{EdRosengen}
Let $G$ be a locally compact group with no nontrivial compact
subgroups, and suppose $G$ has an abelian closed
subgroup of finite index.  Then every nonzero element of
$L^p(G)$, where $1 \leq p \leq 2$, has linearly independent translations.
\end{Thm}
\begin{proof}
Since $G$ has no nontrivial compact subgroups, it is torsion free.
Furthermore for $1 \leq p \leq 2$, if $0 \ne \phi \in \mathbb{C}A$
and $0 \ne f \in L^p(A)$,
then $\phi f \ne 0$ by \cite[Theorem 1.2]{EdgarRosenblatt79}.
The result now follows from Proposition \ref{Pabelianbyfinite}(b).
\end{proof}

\begin{Thm}
Let $G$ be a locally compact abelian group, let $n \in \mathbb{N}$,
and let $1\le p \in \mathbb{R}$.  Assume that $p \le 2n/(n-1)$.
Suppose $G$ has a closed subgroup of finite index
isomorphic to $\mathbb{R}^n$ or $\mathbb{Z}^n$ as a locally compact
abelian group.  Let $H\leqslant G$, let $\theta \in \mathbb{C}H$, and let
$\theta \in \mathbb{C}G$ and let
$0 \ne f \in L^p(G)$.
\end{Thm}
\begin{enumerate}[\normalfont(a)]
\item
If $\theta$ is a nonzero divisor in $\mathbb{C}H$, then $\theta f \ne
0$.
\item
If $H$ is torsion free and $\theta \ne 0$, then $\theta f \ne 0$.
\end{enumerate}

\begin{proof}
We apply Proposition \ref{Pabelianbyfinite} with $A = \mathbb{R}^n$
or $\mathbb{Z}^n$.  We need to check the hypothesis
that if $0 \ne \phi \in \mathbb{C}A$ and $0 \ne
f \in L^p(A)$, then $\phi f \ne 0$.  For the case $A = \mathbb{R}^n$,
this follows from \cite[Theorem 3]{Rosenblatt08}, while for the case
$A = \mathbb{Z}^n$, this follows from \cite[Theorem
2.1]{LinnellPuls01}.
\end{proof}
{\bf Acknowledgments.}  The research of the second author was partially supported by PSC-CUNY grant 66269-00 44.

\bibliographystyle{plain}
\bibliography{dependtranssqintrep}
\end{document}